\documentclass{amsart}
\usepackage{amsmath,amssymb,amsthm,color}
\usepackage{cite}

\numberwithin{equation}{section}

\newtheorem{Theorem}{Theorem}[section]

\newtheorem{Corollary}[Theorem]{Corollary}
\newtheorem{Lemma}[Theorem]{Lemma}
\newtheorem{Proposition}[Theorem]{Proposition}
\newtheorem{Remark}{Remark}[section]

\begin{document}

\title[Lifespan for a system of semilinear heat equations]{
Lifespan of solutions for a weakly coupled system of semilinear heat equations}
\author[K. Fujiwara]{Kazumasa Fujiwara}

\address[K. Fujiwara]{%
Centro di Ricerca Matematica Ennio De Giorgi\\
Scuola Normale Superiore\\
Piazza dei Cavalieri, 3, 56126 Pisa\\
Italy.
}

\email{kazumasa.fujiwara@sns.it}


\author[M. Ikeda]{Masahiro Ikeda}
\address[M. Ikeda]{Department of Mathematics, Faculty of Science and Technology,
Keio University,
3-14-1 Hiyoshi, Kohoku-ku, Yokohama, 223-8522, Japan}
\email{masahiro.ikeda@keio.jp}
\address{Center for Advanced Intelligence Project, RIKEN, Japan}
\email{masahiro.ikeda@riken.jp}

\author[Y. Wakasugi]{Yuta Wakasugi}

\address[Y. Wakasugi]{Department of Engineering for Production and Environment,
Graduate School of Science and Engineering,
Ehime University,
3 Bunkyo-cho, Matsuyama, Ehime, 790-8577,
Japan}%
\email{wakasugi.yuta.vi@ehime-u.ac.jp}

\begin{abstract}
We introduce a straightforward method to analyze the blow-up of solutions to
systems of ordinary differential inequalities,
and apply it to study the blow-up of solutions to a weakly coupled system of
semilinear heat equations.
In particular,
we give upper and lower estimates of the lifespan of the solution
in the subcritical case.
\end{abstract}

\maketitle

\section{Introduction}
In this paper,
we consider a weakly coupled system of semilinear heat equations
	\begin{align}
	\begin{cases}
	\partial_t u - \Delta u = F(u),
	&\mbox{for} \quad
	t \in \lbrack 0, T), \quad x \in \mathbb R^n,\\
	u(0,x) = u_{0}(x),
	&\mbox{for} \quad
	x \in \mathbb R^n
	\end{cases}
	\label{eq:1.1}
	\end{align}
with $n \ge 1$.
Here $u = (u_1, u_2, \ldots, u_k) : \lbrack 0, T) \times \mathbb R^n \to \mathbb R^k$,
is an $\mathbb R^k$-valued unknown function with $k\ge 1$.
The nonlinearity $F$ is defined as
	\[
	F(u) = (F_1(u), F_2(u), \ldots, F_k(u)), \quad
	F_j(u) = |u_{j+1}|^{p_j}, \quad
	p_j \geq 1 \quad (j=1,2,\ldots,k),
	\]
where $u_{k+1}$ is interpreted as $u_1$.
Also,
$u_0 = (u_{0,1}, \ldots, u_{0,k}) \in L^1(\mathbb R^n)\cap L^{\infty}(\mathbb R^n)$
is a given initial data.

The system \eqref{eq:1.1} with $k=2$ was introduced by
Escobedo and Herrero \cite{EsHe91}
as a simple model of a reaction-diffusion system,
which can describe a heat propagation in a two-component combustible mixture.
Later on, many authors studied the system \eqref{eq:1.1} and
determined the so-called critical exponent.
Here the critical exponent is defined in the following way.
Let $\mathcal{P}$ be a $k \times k$ matrix defined by
	\begin{align*}
	\mathcal{P} = \begin{pmatrix}
	0		& p_1		& 0		& \cdots	& 0\\
	0		& 0		& p_2		& \cdots	& 0\\
	\vdots& \vdots	& \ddots	& \ddots	& \vdots\\
	0		& 0		& \cdots	& 0		& p_{k-1}\\
	p_k	& 0		& \cdots	& 0		& 0
	\end{pmatrix}
	\end{align*}
and let $I$ be the $k$-th identity matrix.
If $(p_1,\ldots,p_k) \neq (1,\ldots,1)$, then it is easy to see that
$| \mathcal{P} - I | \neq 0$,
which enables us to define
	\[
	\alpha
	= {}^t\!(\alpha_1,\ldots,\alpha_k) = (\mathcal{P}-I)^{-1} \cdot {}^t\! (1,\ldots,1).
	\]
Here for each $j$, $\alpha_j$ is explicitly given by
	\begin{align}
	\alpha_j = \frac{\sum_{h=0}^{k-2}\prod_{m=0}^h p_{j+m}+1}{\prod_{m=1}^k p_m -1},
	\label{eq:1.2}
	\end{align}
where for $j > k$, $p_j$ is interpreted as $p_{j-k}$.
Let $\alpha_{\mathrm{max}} = \max_{1\le j \le k} \alpha_j$.
It is known that
if $\alpha_{\mathrm{max}} < n/2$,
then the system \eqref{eq:1.1} admits a unique global solution for small initial data;
if $\alpha_{\mathrm{max}} \geq n/2$,
then solutions to \eqref{eq:1.1} blow up in a finite time for any
nontrivial nonnegative initial data.
In this sense, the relation
$\alpha_{\mathrm{max}} = n/2$
is called the critical exponent
(see \cite{EsHe91, EsLe95, MoHu98, Re00, Um03, AoTsYa07, IsKaSi16}).
This is a natural extension of
the pioneering works by
\cite{Fu66, Hay73, KoSiTa77, We81}
for the single semilinear heat equation
because if $k=1$,
then $\alpha_1 = 1/(p-1)$ and the critical case is given when $p = 1+2/n$,
which is the well-known Fujita exponent.

The reason why the exponent
$\alpha$ is related to the critical exponent is explained by the scaling argument
(the following argument is also found in \cite{NiWa15}).
If $u$ is a solution to \eqref{eq:1.1},
then so is $u^{(\lambda)}$
for any $\lambda>0$,
where for $1 \leq j \leq k$,
	\begin{align}
	u_j^{(\lambda)}
	:= \lambda^{2\alpha_j} u_j(\lambda^2 t, \lambda x).
	\label{eq:1.3}
	\end{align}
Moreover, the $L^1$-norm of the initial data
$\| u_j^{(\lambda)}(0,\cdot) \|_{L^1} = \lambda^{2\alpha_j -n} \| u_{0,j} \|_{L^1}$
is invariant under the critical condition $\alpha_j = n/2$.

Here, we remark that
the invariant scaling transformation \eqref{eq:1.3} implies
that when $k=1$, for any $\lambda > 0$,
	\begin{align}
	T_m(u_0) \| u_0 \|_{L^1}^{1/(1/(p-1)-n/2)}
	&= T_m(u_0) \lambda^{-2} \| u^{(\lambda)}(0,\cdot) \|_{L^1}^{1/(1/(p-1)-n/2)}
	\nonumber\\
	&= T_m(u^{(\lambda)}(0,\cdot)) \| u^{(\lambda)}(0,\cdot) \|_{L^1}^{1/(1/(p-1)-n/2)}
	\label{eq:1.4}
	\end{align}
where
	\begin{align}%
	T_m(u_0) &:= \sup \{ T>0 ; \,
	\mbox{With the initial data $u_0$, there exists a unique solution}\,
	\nonumber\\
		&\qquad\qquad\qquad \qquad u \in C([0,T);
		L^1(\mathbb R^n) \cap L^{\infty}(\mathbb R^n))
	\, \mbox{to \eqref{eq:1.1}} \}.
	\label{eq:1.5}
	\end{align}%
Since $\lambda > 0$ is arbitrary, from \eqref{eq:1.4},
when $k=1$,
it is expected that for some $0 < c < C$,
	\begin{align}
	c \| u_0 \|_{L^1}^{-1/(1/(p-1)-n/2)}
	\leq T_m(u_0)
	\leq C \| u_0 \|_{L^1}^{-1/(1/(p-1)-n/2)}.
	\label{eq:1.6}
	\end{align}
Indeed, the first estimate of \eqref{eq:1.6} holds for any $n \ge 1$
and non-negative $u_0 \in L^1(\mathbb{R}^n) \cap L^{\infty}(\mathbb{R}^n)$ when $p<p_F$,
provided that $\| u_0 \|_{L^1}$ is replaced by
$\| u_0 \|_{L^1\cap L^{\infty}}$.
For details, see Proposition \ref{Proposition:1.3} below.

The aim of this paper is
to prove the blow-up of solutions to \eqref{eq:1.1}
by a straightforward approach of ordinary differential equation(ODE).
Specifically, a blow-up of solutions to \eqref{eq:1.1} follows from the study of
the following ODE system:
	\begin{align}
	\frac{d}{dt} f(t) = \widetilde F(t,f(t)),
	\qquad
	\mbox{for} \quad 0 < t < T.
	\label{eq:1.7}
	\end{align}
A general approach to study ODE systems is to find some function
$G: \mathbb R^k \to \mathbb R$
such that a single ODE of $G$ follows from \eqref{eq:1.7}.
In particular, we may find a function $G$ satisfying that
	\begin{align}
	\frac{d}{dt} G(f(t)) \geq C G(f(t))^\gamma,
	\label{eq:1.8}
	\end{align}
with some positive constants $C$ and $\gamma$,
which may imply that $G$ blows up at a finite time.
For example, Mochizuki \cite{Mo98} showed that
solutions to \eqref{eq:1.1} blow up
when $k=2$, by studying \eqref{eq:1.7}
with $\widetilde F = F$ and $G(f) = f_1$.
Indeed, Mochizuki obtained an ODE for $f_1$ by connecting two ODEs of \eqref{eq:1.7}
with the following identity:
	\begin{align}
	\frac{1}{p_2+1} \frac{d}{dt} (f_1^{p_2+1})
	= f_1^{p_2} f_2^{p_1}
	= \frac{1}{p_1+1} \frac{d}{dt} (f_2^{p_1+1}).
	\label{eq:1.9}
	\end{align}
By combining \eqref{eq:1.7} and \eqref{eq:1.9},
one formally has
	\begin{align}
	\frac{d}{dt} f_1(t)
	= f_2(t)^{p_1}
	\sim f_1(t)^{\frac{p_1(p_2+1)}{p_1+1}}
	= f_1(t)^{\frac{1}{\alpha_1}+1},
	\label{eq:1.10}
	\end{align}
from which a sharp lifespan estimate for \eqref{eq:1.1} is obtained.
Here we say that a lifespan estimate is sharp if
$p_j = p < p_F$ and $u_{0,j} = v_0 \in L^1(\mathbb R^n) \cap L^\infty(\mathbb R^n)$ for any $j$,
then $T_m(v_0)$ satisfies \eqref{eq:1.6}.
However, for $k \geq 3$, there is no identity like \eqref{eq:1.9}
which unites equations of \eqref{eq:1.7}.
In the case where $k \geq 3$,
Fila and Quittner \cite{FiQu99} studied \eqref{eq:1.1} by an ordinary differential inequality \eqref{eq:1.8}
with $G(f) = \sum_{j=1}^k f_j$
and discussed the blow-up rate of solutions near blow-up time
(See Lemma 2.1 in \cite{FiQu99}).
However, with this choice of $G$,
it seems difficult to obtain the sharp lifespan estimate for \eqref{eq:1.7}
and consequently also for \eqref{eq:1.1} as well.
On the other hand,
Wang \cite{Wan02} obtained the blow-up rate of solutions to \eqref{eq:1.1}
on a bounded domain in $\mathbb{R}^n$ with $k \geq 3$
by \eqref{eq:1.8} with
$G(f) = \prod_{j=1}^k f_j$
(See the proof of Theorem 1 in \cite{Wan02}).
However,
it also seems difficult to get the sharp lifespan estimate
with $G(f) = \prod_{j=1}^k f_j$.

In this paper,
so as to obtain a sharp lifespan estimate,
we avoid using a function $G$
such as $G(f) = \prod_{j=1}^k f_j$ or $G(f) = \sum_{j=1}^k f_j$.
On the other hand,
we introduce a concatenation of equations of \eqref{eq:1.7}
derived by weakly coupled interaction
and show that for each $j$,
	\[
	\frac{d}{dt} f_j(t) \geq C f_j(t)^{\frac{1}{\alpha_j}+1}.
	\]
This is a natural extension of the idea of \eqref{eq:1.10}.

Before stating our main results,
we recall the local existence of the solution.
\begin{Proposition}
\label{Proposition:1.1}
For any $p_j \ge 1 (j=1,\ldots,k)$
and
$u_0 \in L^1(\mathbb R^n) \cap L^{\infty}(\mathbb R^n)$,
there exist $T>0$ and a unique solution
$u \in C([0,T); L^1(\mathbb R^n) \cap L^{\infty}(\mathbb R^n))$
to \eqref{eq:1.1},
which satisfies \eqref{eq:1.1} in the classical sense for $0 < t < T$.
\end{Proposition}
Proposition \ref{Proposition:1.1} may be obtained
by a simple modification of the proof of Theorem 2.1 in \cite{EsHe91}.
Next we define lifespan of solutions to \eqref{eq:1.1},
in a similar manner to \eqref{eq:1.5},
by
	\begin{align*}%
	T_m &:= \sup \{ T>0 ; \, \mbox{There exists a unique solution},
	\\
		&\qquad\qquad\qquad \qquad u \in C([0,T);
		L^1(\mathbb R^n) \cap L^{\infty}(\mathbb R^n)))
	\, \mbox{to \eqref{eq:1.1}} \}.
	\end{align*}%

As a corollary of Proposition \ref{Proposition:1.1},
we have the following blow-up alternative.
\begin{Corollary}
If $T_m < \infty$,
then we have $\lim_{t\to T_m -0} \| u(t) \|_{L^1\cap L^{\infty}} = \infty$.
\end{Corollary}

Moreover, from Proposition \ref{Proposition:1.1} and a standard a priori estimate,
we have the lower bound of the lifespan.
\begin{Proposition}\label{Proposition:1.3}
In addition to the assumptions of Proposition \ref{Proposition:1.1},
we further assume $(p_1, \ldots, p_k) \neq (1,\ldots,1)$.
Then, there exist constants $\varepsilon_0 > 0$ and $C>0$ such that
for any $u_0$ satisfying $\| u_0 \|_{L^1\cap L^{\infty}} \le \varepsilon_0$,
the lifespan satisfies
\begin{align*}%
	T_m \ge
	\begin{cases}
		\infty &(\alpha_{\mathrm{max}} < n/2),\\
		C \| u_0 \|_{L^1\cap L^{\infty}}^{-1/(\alpha_{\mathrm{max}} - n/2)}
			&(\alpha_{\mathrm{max}} > n/2).
	\end{cases}
\end{align*}%
\end{Proposition}
We give a proof of Proposition \ref{Proposition:1.3} in Section 4.

\begin{Remark}
In the critical case where $\alpha_{\mathrm{max}} = n/2$,
we can also have the estimate
$T_m \ge \exp \left( C \| u_0 \|_{L^1 \cap L^{\infty}}^{-(\min_j p_j -1)} \right)$
in the same way.
However, it seems not optimal and
we do not pursue here the critical case.
\end{Remark}

In order to state the main blow-up result of this paper,
we introduce some notation.
Let $\lambda >0$ and nonnegative $\phi$ satisfy
$\| \psi \|_{L^1} = 1$ and
	\[
	\frac \lambda 2 \psi(x) =
	\begin{cases}
	- \Delta \psi(x) > 0,
	&\mathrm{if} \quad |x| \leq 1,\\
	0,
	&\mathrm{if} \quad |x| \geq 1.
	\end{cases}
	\]
Namely,
$\lambda/2$
is a positive eigenvalue of Dirichlet Laplacian on the unit disc,
and
$\psi$ is a null-extension of normalized positive eigenfunction.
We put $\phi_R(x)$ as $\psi(x/R)^2$ with $R>0$.
Then
	\[
	\begin{cases}
	R^{-2} \lambda \phi_R(x) \geq
	- \Delta \phi_R(x),
	&\mathrm{if} \quad |x| \leq R,\\
	\phi_R(x) = \nabla \phi_R(x) = 0,
	&\mathrm{if} \quad |x| \geq R.
	\end{cases}
	\]
At last, for solutions to \eqref{eq:1.1},
we define
	\[
	U_{j,R}(t) = \int_{\mathbb R^n} u_{j}(t,x) \phi_R(x) dx\quad (j=1,\ldots,k).
	\]

Then we have the following estimate.

\begin{Theorem}
\label{Theorem:1.4}
Let $p_j \geq 1$ for any $ 1 \leq j \leq k$
but let $p_j > 1$ for some $j$.
Let $u_{0} \in L^1(\mathbb R^n) \cap L^\infty(\mathbb R^n)$ satisfy that
$u_{0,j}$ is nonnegative for any $j$.
Let $u \in C([0,T_m);L^1(\mathbb R^n) \cap L^\infty(\mathbb R^n))$
be the corresponding solution of \eqref{eq:1.1}.
If there exists $j_0 \in \{1,\ldots,k\}$ such that
$u_{0,j_0} \not \equiv 0$ and
	\begin{align}
	\alpha_{j_0}
	&> \frac{n}{2},
	\label{eq:1.11}
	\end{align}
then $u$ cannot exist globally and
there exists $R_0 > 0$ such that
	\begin{align}
	T_m
	\le C_0 U_{j_0,R_0}(0)^{-1/(\alpha_{j_0} - n/2)}.
	\label{eq:1.12}
	\end{align}
with some constant $C_0$.
\end{Theorem}

\begin{Remark}
{\rm (i)}
$R_0$ is determined by
	\begin{align}
	U_{j_0,R_0}(0)
	= C_1 R_0^{-2 \alpha_{j_0} + n}
	\label{eq:1.13}
	\end{align}
with a certain constant $C_1$.
\eqref{eq:1.11} implies existence of $R_0$ satisfying \eqref{eq:1.13}
for any nonnegative $u_{0,j_0} \in L^1(\mathbb R^n) \backslash \{0\}$,
since $f : R \mapsto U_{j_0,R}(0)$ is an increasing function for $R \geq 0$
	with $\lim_{R \to -0} f(R) = 0$.
\newline
{\rm (ii)}
Theorem \ref{Theorem:1.4} seems to be sharp from the viewpoint of the scaling
(see \eqref{eq:1.6}).
\newline
{\rm (iii)}
The proof of Theorem \ref{Theorem:1.4} implies that
	\[
	U_{j,R_0}(t) \geq C_{j} (T_0-t)^{-\alpha_j}
	\]
for any $1 \leq j \leq k$
when $t \in (0,T_0)$ close to $T_0$,
where $C_{j}$ is a positive constant and
$T_0$ is the RHS of \eqref{eq:1.12}.
However, Theorem \ref{Theorem:1.4} does not give the exact blow-up rate
because $T_0$ is nothing but an upper bound of blow-up time
(see \cite{Ba77}).
\end{Remark}

In Section 2,
we prepare blow-up results for a system of ODEs.
Then, in Section 3,
we show Theorem \ref{Theorem:1.4}
by combining a test function method developed by
\cite{BaPi85, FuOz16, FuOz17, FuIkWa17, IkeFuWapre, MiPo01, Zh99}
with the ODE argument discussed in Section 2.
Section 4 is devoted to the proof of Proposition \ref{Proposition:1.3}.

\section{ODE argument}
In this section, for $k \geq 2$,
$p_j \ge 1 \ (j=1,\ldots,k)$ with $(p_1,\ldots, p_k) \neq (1,\ldots,1)$,
and $\widetilde{\lambda} \ge 0$,
we consider the following ODE system:
	\begin{align}
	\begin{cases}
	\displaystyle
	\frac{d}{dt} f_j(t) \geq \widetilde C_j f_{j+1}(t)^{p_j}, \quad
	& \mathrm{for} \quad 1 \leq j \leq k-1, \quad  t \in \lbrack 0, T),\\[8pt]
	\displaystyle
	\frac{d}{dt} f_k(t) \geq \widetilde C_k e^{-\widetilde \lambda t} f_{1}(t)^{p_k}, \quad
	& \mathrm{for} \quad t \in \lbrack 0, T).
	\end{cases}
	\label{eq:2.1}
	\end{align}
where $\widetilde C_j > 0$ for any $j$.

In order to state the blow-up statement for \eqref{eq:2.1},
we introduce the following notation.
For $1 \leq j < k$, let
	\[
	P_{k-j} = \sum_{h=0}^{j} \prod_{m = 0}^h p_{k-j+m} + 1,
	\qquad
	Q_{k-j} = \sum_{h=0}^{j-1} \prod_{m = 0}^h p_{k-j+m} + 1,
	\]
and let
	\[
	P_{k} = p_k + 1,
	\qquad
	Q_k = 1.
	\]
Then
$(P_{k-j})_{j=1}^{k-1}$ and $(Q_{k-j})_{j=1}^{k-1}$ satisfy that for any $1 \leq j < k$,
	\begin{align}
	&P_{k-j} = p_{k-j} P_{k-j+1} + 1,
	\label{eq:2.2}\\
	&Q_{k-j} = p_{k-j} Q_{k-j+1} + 1,
	\label{eq:2.3}
	\end{align}
and
	\begin{align}
	&P_{k-j} - Q_{k-j} = \prod_{m = k-j}^k p_{\ell}.
	\label{eq:2.4}
	\end{align}
By \eqref{eq:1.2}, \eqref{eq:2.4}, and the definition of $Q_1$ and $P_2$,
	\begin{align}
	&\alpha_1 = \frac{Q_1}{P_1-Q_1-1},
	\qquad
	\alpha_2 = \frac{P_2}{P_1-Q_1-1}.
	\label{eq:2.5}
	\end{align}
For $1 \leq j < k$, we put
	\begin{align*}
	A_{k-j} &= P_{k-j}^{-1} \widetilde C_{k-j}
	\prod_{h=0}^{j-1} (P_{k-h}^{-1} \widetilde C_{k-h})^{\prod_{\ell=k-j}^{k-h-1}p_\ell},\\
	L_{k-j} &= \widetilde \lambda \prod_{m=k-j}^{k-1} p_m.
	\end{align*}
and let
	\[
	A_k = P_k^{-1} \widetilde C_k,
	\qquad
	L_{k} = \widetilde \lambda.
	\]
Then for $ 1 \leq j \leq k-1$,
	\begin{align}
	A_{k-j} &= P_{k-j}^{-1} \widetilde C_{k-j} A_{k-j+1}^{p_{k-j}},
	\label{eq:2.6}\\
	L_{k-j} &= p_{k-j} L_{k-j+1}.
	\label{eq:2.7}
	\end{align}
Moreover, let
	\begin{align}
	\label{eq:2.8}
	\widetilde C
	= L_1^{-1} (P_{1}-Q_1 -1) 2^{-p_1(P_2-1)/Q_1} P_1^{1/Q_1} A_{1}^{1/Q_1}.
	\end{align}
Now we are in the position to state our blow-up statement for \eqref{eq:2.1}.

\begin{Proposition}
\label{Proposition:2.1}
If $f_j(0) \geq 0$ for $j \neq 2$ and
	\begin{align}
	f_2 (0)
	> \widetilde C^{-\alpha_2} A_{2}^{\alpha_2/(\alpha_1 P_{2})}
	\widetilde C_1^{-\alpha_2 Q_2 /(\alpha_1 P_2)},
	\label{eq:2.9}
	\end{align}
then solutions $f_j$ to \eqref{eq:2.1} satisfy that
	\begin{align}
	f_1(t)
	&\geq \widetilde C^{-1/\alpha_1}
	( e^{- L_1 t / Q_1} - e^{- L_1 \widetilde T_0 / Q_1})^{-\alpha_1}
	- A_{2}^{-1/P_{2}} \widetilde C_1^{Q_2/P_2} f_2(0)^{Q_1/P_2},
	\label{eq:2.10}
	\end{align}
where
	\begin{align*}
	\widetilde T_0 &= - L_1^{-1} Q_1
	\log( 1 - \widetilde C^{-1} A_{2}^{1/(\alpha_1 P_{2})}
	\widetilde C_1^{-Q_2 /(\alpha_1 P_2)} f_2(0)^{- 1/\alpha_2 } ).
	\nonumber
	\end{align*}
\end{Proposition}

\begin{Remark}
Under the condition \eqref{eq:2.9},
we have
$\widetilde T_0 > 0$.
\end{Remark}

\begin{Remark}
\label{Remark:2.2}
Proposition \ref{Proposition:2.1} implies that
if $t$ is close to $\widetilde T_0$,
the minorant of $f_1(t)$ takes the form of $(\widetilde T_0-t)^{-\alpha_1}$.
Then if $f_k$ has a sense when $t$ is close to $\widetilde T_0$,
we have
	\[
	\frac{d}{dt} f_k(t) \geq C (\widetilde T_0-t)^{-p_k \alpha_1},
	\]
and therefore $f_k$ may satisfy
	\[
	f_k(t) \geq C (\widetilde T_0-t)^{- \alpha_k}
	\]
for $t$ close to $\widetilde T_0$.
Indeed,
	\begin{align*}
	p_k \alpha_1 - 1
	&= p_k \frac{\sum_{h=0}^{k-2} \prod_{\ell = 0}^h p_{k+1+\ell} + 1}%
	{\prod_{\ell = 1}^k p_{\ell} - 1} -1\\
	&= \frac{\sum_{h=0}^{k-3} \prod_{\ell = 0}^{h+1} p_{k+\ell} + p_k + 1}%
	{\prod_{\ell = 1}^k p_{\ell} - 1}\\
	&= \frac{\sum_{h=0}^{k-2} \prod_{\ell = 0}^{h} p_{k+\ell} + 1}%
	{\prod_{\ell = 1}^k p_{\ell} - 1}
	= \alpha_k,
	\end{align*}
where we put $p_{k+j} = p_j$ for $1 \leq j < k$.
Similarly, we have
	\[
	f_j(t) \geq C (\widetilde T_0-t)^{- \alpha_j}
	\]
for $t$ close to $\widetilde T_0$ as long as $f_j$ has a sense.
\end{Remark}

In order to prove Proposition \ref{Proposition:2.1},
at first,
we recall the comparison principle for ODE systems.

\begin{Lemma}[\!\!\cite{Ka32}]
\label{Lemma:2.2}
Let $\mathcal F_j : \lbrack 0, T) \times \mathbb R \to \mathbb R$
satisfy that if $x > y$,
	\[
	\mathcal F_j(t,x) > \mathcal F_j(t,y)
	\]
holds for any $1 \leq j \leq k$ and $ t \in \lbrack 0,T)$.
Let $(f_j)_{j=1}^k, (g_j)_{j=1}^k \subset C^1([0,T) ; \mathbb{R})$ satisfy that
for any $1 \leq j \leq k$,
	\[
	\frac{d}{dt} f_j(t) \geq \mathcal F_j(t, f_{j+1}(t)),
	\qquad
	\frac{d}{dt} g_j(t) \leq \mathcal F_j(t, g_{j+1}(t)),
	\qquad
	\mbox{for}
	\quad t \in \lbrack 0, T)
	\]
and
	\[
	f_j(0) \geq g_j(0),
	\]
where $f_{k+1} = f_1$ and $g_{k+1} = g_1$.
If $f_l(0) > g_l(0)$ for some $1 \leq l \leq k$,
then $f_j(t) > g_j(t)$
holds 
for any $t \in (0,T)$ and $1 \leq j \leq k$.
\end{Lemma}

For completeness, we prove Lemma \ref{Lemma:2.2}.

\begin{proof}
Without loss of generality, we assume $f_k(0) > g_k(0)$.
This implies
$f_k(t) > g_k(t)$
for $t \in \lbrack 0,\tau_0)$ with some $\tau_0 > 0$.
Therefore, the inequality
$\mathcal{F}_{k-1}(t, f_k(t)) > \mathcal{F}_{k-1}(t, g_k(t))$
holds
for $t \in (0, \tau_0)$
and this leads to
	\[
	\frac{d}{dt} ( f_{k-1}(t) - g_{k-1}(t) )
	\geq \mathcal F_{k-1}(t, f_{k}(t)) - \mathcal F_{k-1}(t, g_{k}(t)) > 0
	\]
for $t \in (0,\tau_0)$.
Hence, the inequality above and $f_{k-1}(0) \ge g_{k-1}(0)$ imply that
$f_{k-1} (t) > g_{k-1}(t)$ holds for $t \in (0,\tau_0)$.
Repeating this argument, we see that
$f_{j}(t) > g_{j}(t)$ holds for any $j$ and $t \in (0,\tau_0)$.
We suppose that for some $j_0 \in \{ 1,\ldots, k \}$ and some $t \in (0,T)$,
$f_{j_0}(t) = g_{j_0}(t)$ holds.
Then, we define
	\[
	\tau_{j_0} = \inf\{ 0 < t < T ;  f_{j_0}(t) = g_{j_0}(t) \}.
	\]
We note that $\tau_{j_0} > \tau_0$.
By the continuity of $f_{j_0}$ and $g_{j_0}$,
	\[
	\frac d{dt} f_{j_0}(\tau_{j_0}) - \frac d{dt} g_{j_0}(\tau_{j_0})
	\leq 0.
	\]
Since
	\[
	\frac{d}{dt} (f_{j_0}(t) - g_{j_0}(t))
	\ge \mathcal{F}_{j_0}(t,f_{j_0+1}(t)) - \mathcal{F}_{j_0}(t,g_{j_0+1}(t)),
	\]
we have $f_{j_0+1}(\tau_{j_0}) \leq g_{j_0+1}(\tau_{j_0})$.
Now, we can also define
	\[
	\tau_{j_0 + 1} = \inf\{ 0 < t < T ;  f_{j_0+1} (t) = g_{j_0+1} (t) \}
	\]
and we obtain
$\tau_{j_0+1} < \tau_{j_0}$.
Repeating this procedure, we define
$\tau_j = \inf \{ 0<t<T ; f_j(t) = g_j(t)\}$
satisfying that
$\tau_{j_0} > \tau_{j_0+1} > \cdots > \tau_k > \tau_1 > \cdots > \tau_{j_0-1}$
(when $j_0 = 1$, $\tau_1 > \cdots > \tau_k$).
However, the same argument implies $\tau_{j_0-1} > \tau_{j_0}$
(when $j_0=1$, $\tau_k > \tau_1$),
which is a contradiction and we have the assertion.
\end{proof}

For the proof of Proposition \ref{Proposition:2.1},
the next lemma plays a critical role.
For simplicity, hereinafter, we denote $\frac{d}{dt} f$ as $f'$.

\begin{Lemma}
\label{Lemma:2.3}
Let $T>0$ and $C, \widetilde \lambda, p,q \geq 0$.
Let $f \in C^2(\lbrack0, T); \lbrack 0, \infty) )$
and $g \in C^1(\lbrack0, T); \lbrack 0, \infty) )$ satisfy that
	\[
	\begin{cases}
	\displaystyle
	C f(t)^p
	\leq e^{\widetilde \lambda t} g'(t)
	f'(t)^q,
	&\mbox{for} \quad t \in \lbrack 0 ,T),\\
	\displaystyle
	f'(t),\ f''(t) \geq 0,
	&\mbox{for} \quad t \in \lbrack 0 ,T).
	\end{cases}
	\]
Then
	\[
	\frac{C}{p+1} ( f(t)^{p+1} - f(0)^{p+1} )
	\leq e^{\widetilde \lambda t} g(t) f'(t)^{q+1} - g(0) f'(0)^{q+1}.
	\]
\end{Lemma}

\begin{proof}
By using integration by parts,
	\begin{align*}
	&\frac{C}{p+1} ( f(t)^{p+1} - f(0)^{p+1} )\\
	&= C \int_0^t f(\tau)^{p} f'(\tau) d\tau\\
	&\leq \int_0^t e^{\widetilde \lambda \tau} g'(\tau) f'(\tau)^{q+1} d\tau\\
	&= e^{\widetilde \lambda t} g(t) f'(t)^{q+1} - g(0) f'(0)^{q+1}
	- \int_0^t g(\tau) ( e^{\widetilde \lambda \tau} f'(\tau)^{q+1})' d\tau\\
	&\leq e^{\widetilde \lambda t} g(t) f'(t)^{q+1} - g(0) f'(0)^{q+1}.
	\end{align*}
\end{proof}

\begin{proof}[Proof of Proposition \ref{Proposition:2.1}]
By Lemma \ref{Lemma:2.2},
it is enough to show \eqref{eq:2.10} for
	\[
	\begin{cases}
	g_j'(t) = \widetilde C_j g_{j+1}(t)^{p_j}, \quad
	& \mathrm{for} \quad 1 \leq j \leq k-1, \quad  t \in \lbrack 0, T),\\
	g_k'(t) = \widetilde C_k e^{-\widetilde \lambda t} g_{1}(t)^{p_k}, \quad
	& \mathrm{for} \quad t \in \lbrack 0, T),
	\end{cases}
	\]
$g_j(0) = 0$ for $j \neq 2$ and
	\begin{align}
	\widetilde C^{-\alpha_2} A_{2}^{\alpha_2/(\alpha_1 P_{2})}
	\widetilde C_1^{-\alpha_2 Q_2 /(\alpha_1 P_2)}
	< g_2(0) < f_2(0).
	\label{eq:2.11}
	\end{align}
Again by Lemma \ref{Lemma:2.2},
we remark that
$g_j(t) \geq 0$ for any $1 \leq j \leq k$.
By Lemma \ref{Lemma:2.3} and $\widetilde C_k g_1^{p_k} = e^{\widetilde \lambda t} g_k'$,
we deduce
	\[
	\frac{\widetilde C_k}{p_k+1} g_1(t)^{p_{k}+1}
	\leq e^{\widetilde \lambda t} g_k(t) g_1'(t),
	\]
which is rewritten as
	\begin{align}
	A_k g_1(t)^{P_k}
	\le e^{L_k t} g_k(t) g_1'(t)^{Q_{k}}.
	\label{eq:2.12}
	\end{align}
Here, we have used the fact that
	\[
	g_1''(t)
	= p_1 \widetilde{C}_1 g_2(t)^{p_1-1} g_2'(t)
	= p_1 \widetilde{C}_1 \widetilde{C}_2 g_2(t)^{p_1-1} g_3(t)^{p_2}
	\geq 0.
	\]
Taking the $p_{k-1}$-th power in both sides of \eqref{eq:2.12},
multiplying the both sides by $\widetilde{C}_{k-1}$,
and by \eqref{eq:2.2},  \eqref{eq:2.3}, \eqref{eq:2.6}, \eqref{eq:2.7},
and using $g_{k-1}'(t) = \widetilde C_{k-1} g_k(t)^{p_{k-1}}$, we have
	\[
	P_{k-1} A_{k-1} g_1(t)^{P_{k-1}-1}
	\leq e^{L_{k-1} t} g_{k-1}'(t) g_1'(t)^{Q_{k-1} -1}.
	\]
Again, by Lemma \ref{Lemma:2.3},
we see that
	\[
	A_{k-1} g_1(t)^{P_{k-1}}
	\leq e^{ L_{k-1} t} g_{k-1}(t) g_1'(t)^{Q_{k-1}}.
	\]
Repeating this argument, we have
	\begin{align*}
	A_{2} g_1(t)^{P_{2}}
	&\leq e^{ L_{2} t} g_{2}(t) g_1'(t)^{Q_2}
	- g_{2}(0) g_1'(0)^{Q_2}\\
	&= e^{ L_{2} t} g_{2}(t) g_1'(t)^{Q_2} - \widetilde C_1^{Q_2} g_{2}(0)^{Q_1}.
	\end{align*}
Here, we put
	\[
	\overline g_1(t)
	= g_1(t) + A_{2}^{-1/P_{2}} \widetilde C_1^{Q_2/P_2} g_2(0)^{Q_1/P_2}.
	\]
Then, we have
	\begin{align}%
	\label{eq:2.13}
	A_2 \overline{g}_1(t)^{P_2}
	&= A_2 \Big( g_1(t)
	+ A_2^{-1/P_2} \widetilde{C}_1^{Q_2/P_2}g_2(0)^{Q_1/P_2} \Big)^{P_2}\\
	\nonumber
	&\le 2^{P_2-1} \Big( A_2 g_1(t)^{P_2} + \widetilde{C}_1^{Q_2} g_2(0)^{Q_1} \Big)\\
	\nonumber
	&\le 2^{P_2-1} e^{L_2 t} g_2(t) \overline{g}_1'(t)^{Q_2}.
	\end{align}%
Then, taking the $p_{1}$-th power in both sides of \eqref{eq:2.13},
and multiplying the both sides by $\widetilde{C}_1$,
we have
	\[
	2^{-p_1(P_2-1)} P_1 A_{1} \overline g_1(t)^{P_{1}-1}
	\leq e^{L_1 t} \overline g_1'(t)^{Q_1}.
	\]
Therefore, we obtain
	\[
	\overline g_1'(t)
	\geq
	2^{-p_1(P_2-1)/Q_1} P_1^{1/Q_1} A_{1}^{1/Q_1}
	e^{-L_1 t/Q_1} \overline g_1(t)^{(P_{1}-1)/Q_1},
	\]
which and \eqref{eq:2.5}, \eqref{eq:2.8}, and $(P_1-1)/Q_1 > 1$ imply that
	\begin{align}
	\label{eq:2.14}
	\overline g_1(t)
	&\geq \Big( \overline g_1(0)^{-(P_{1}-Q_1-1)/Q_1}
	- \widetilde C (1 - e^{-L_1 t /Q_1} ) \Big)^{-Q_1/(P_{1} - Q_1 -1)}\\
	\nonumber
	&= \Big( \overline g_1(0)^{-1/\alpha_1}
	- \widetilde C (1 - e^{-L_1 t /Q_1} ) \Big)^{-\alpha_1}\\
	\nonumber
	&= \widetilde C^{-1/\alpha_1}
	( e^{-L_1 t /Q_1}  - e^{-L_1 \widetilde T_1 /Q_1} )^{-\alpha_1},
	\end{align}
where
	\[
	\widetilde T_1
	= - L_1^{-1} Q_1 \log
	\Big( 1 - \widetilde{C}^{-1} A_2^{1/(\alpha_1 P_2)}
	\widetilde{C}_1^{-Q_2/(\alpha_1 P_2)}
	g_2(0)^{-1/\alpha_2} \Big).
	\]
Indeed, we compute
	\begin{align*}
	\overline g_1(0)^{-1/\alpha_1}
	&= A_{2}^{1/(\alpha_1 P_{2})} \widetilde C_1^{- Q_2/(\alpha_1 P_2)}
	g_2(0)^{- Q_1 /(\alpha_1 P_2)}\\
	&= A_{2}^{1/(\alpha_1 P_{2})} \widetilde C_1^{- Q_2/(\alpha_1 P_2)}
	g_2(0)^{- (P_1-Q_1-1) /P_2 }\\
	&= A_{2}^{1/(\alpha_1 P_{2})} \widetilde C_1^{- Q_2/(\alpha_1 P_2)}
	g_2(0)^{- 1/\alpha_2 }.
	\end{align*}
We also remark that \eqref{eq:2.11} implies positiveness of $\widetilde T_1$.
From \eqref{eq:2.14}, the definition of $\overline{g}_1(t)$,
and Lemma \ref{Lemma:2.2},
we have
	\[
	f_1(t)
	\geq g_1(t)
	\geq \widetilde{C}^{-1/\alpha_1}
	\Big( e^{-L_1 t/Q_1} - e^{-L_1 \widetilde{T}_1/Q_1} \Big)^{-\alpha_1}
	- A_2^{-1/P_2} \widetilde{C}_1^{Q_2/P_2} g_2(0)^{Q_1/P_2}.
	\]
Finally, by taking the limit $g_2(0) \to f_2(0)$,
the RHS of the inequality above converges to that of \eqref{eq:2.10}.
We note that
$g_2(0) < f_2(0)$ leads to $\widetilde T_1 > \widetilde T_0$,
and
$\widetilde T_1 \to \widetilde T_0$
by letting $g_2(0) \to f_2(0)$.
\end{proof}

\section{Proof of Theorem \ref{Theorem:1.4}}
In this section,
we prove Theorem \ref{Theorem:1.4}.
Here we restate Theorem \ref{Theorem:1.4}
with more details.

\begin{Proposition}
\label{Proposition:3.1}
Let $u_0 \in L^1 (\mathbb R^n) \cap L^{\infty} (\mathbb R^n)$ satisfy
$u_{0,j} \geq 0$ for any $1 \leq j \leq k$.
We assume $p_j \ge 1 \ (j=1,\ldots,k)$ and $(p_1,\ldots,p_k) \neq (1,\ldots,1)$.
Let $u \in C([0,T) ; L^1 (\mathbb R^n) \cap L^{\infty} (\mathbb R^n))$
be a classical solution of \eqref{eq:1.1}.
We further assume that
there exists $j_0 \in \{1,\ldots,k\}$ such that
$\alpha_{j_0} > \frac{n}{2}$ and $u_{0,j_0} \not \equiv 0$.
Then, for some constants $\overline C_1, \overline C_2, \overline C_3$,
we have
\begin{align}
\label{eq:3.1}
	&U_{j_0-1,R_0}(t) + \overline C_2\\
	\nonumber
	&\geq \overline C_1 e^{-2 R_0^{-2} t}
	( e^{- 2 \lambda \prod_{m \neq j_0-2} p_{m} R_0^{-2} t / \alpha_{j_0-1}}
	- e^{- 2 \lambda \prod_{m \neq j_0-2} p_{m} R_0^{-2} T_0 / \alpha_{j_0-1}}
	)^{-\alpha_{j_0-1}}
\end{align}
with $T_0$ satisfying
	\begin{align}
	0 < T_0
	\leq \overline C_3 U_{j_0,R}(0)^{- 1/(\alpha_{j_0} - n/2)}
	\label{eq:3.2}
	\end{align}
and $R_0$ defined by
	\begin{align}
	U_{j_0,R_0}
	= 2^{\alpha_{j_0}} \overline{C_1}^{2 \alpha_{j_0} - n} R_0^{-2 \alpha_{j_0} + n},
	\label{eq:3.3}
	\end{align}
where we interpret $U_{0,R} = U_{k,R}$, $\alpha_{0} = \alpha_k$,
$p_0 = p_k$, and $p_{-1}=p_{k-1}$.
\end{Proposition}

Here, we remark that Proposition \ref{Proposition:3.1}
implies $T_m \le T_0$, that is, the assertion of Theorem \ref{Theorem:1.4}.

\begin{proof}
Without loss of generality,
we assume $j_0 = 2$.
Recall $\| \phi \|_{L^1} = 1$.
Then by the H\"older inequality, we have
	\begin{align}
	\label{eq:3.4}
	\int_{\mathbb R^n} u_{j+1}(t,x) \phi_{R_0}(x) dx
	&\leq
	\bigg( \int_{\mathbb R^n} \phi_{R_0}(x) dx \bigg)^{1/p_j'}
	\bigg( \int_{\mathbb R^n} |u_{j+1}(t,x)|^{p_j} \phi_{R_0}(x) dx \bigg)^{1/p_j}\\
	\nonumber
	&\leq
	R_0^{n/p_j'}
	\bigg( \int_{\mathbb R^n} |u_{j+1}(t,x)|^{p_j} \phi_{R_0}(x) dx \bigg)^{1/p_j}.
	\end{align}
Here, $p_j'$ is the H\"older conjugate of $p_j$, that is,
$p_j' = p_j/(p_j-1)$ if $p_j >1$ and $p_j'=\infty$ if $p_j =1$.
Hereafter, we interpret that $1/p_j' = 0$ if $p_j'=\infty$.
We deduce that
	\begin{align*}%
	U_{j,R_0}'(t)
	&= \int_{\mathbb R^n} \partial_t u_j(t,x) \phi_{R_0}(x) dx\\
	&= \int_{\mathbb R^n} \Delta u_j(t,x) \phi_{R_0}(x) dx
	+ \int_{\mathbb R^n} |u_{j+1}(t,x)|^{p_j} \phi_{R_0}(x) dx \\
	&= \int_{\mathbb R^n} u_j(t,x) \Delta \phi_{R_0}(x) dx
	+ \int_{\mathbb R^n} |u_{j+1}(t,x)|^{p_j} \phi_{R_0}(x) dx \\
	&\geq - 2 \lambda R_0^{-2} \int_{\mathbb{R}^n} u_j(t,x) \phi_{R_0}(x) dx
	+ \int_{\mathbb R^n} |u_{j+1}(t,x)|^{p_j} \phi_{R_0}(x) dx \\
	&\geq - 2 \lambda R_0^{-2} U_{j,R_0}(t)
	+ \int_{\mathbb R^n} |u_{j+1}(t,x)|^{p_j} \phi_{R_0}(x) dx.
	\end{align*}%
Combining this and \eqref{eq:3.4}, we see that
for any $j \in \{1,\ldots,k\}$, $U_{j,R_0}$ satisfies
	\[
	U_{j,R_0}' + 2 \lambda R_0^{-2} U_{j,R_0}
	\geq R_0^{-n(p_j-1)} U_{j+1,R_0}^{p_j}.
	\]
Furthermore, since $U_{j,R} \ge 0$,
we immediately obtain from the above inequality that
	\begin{align}%
	U_{j,R_0}' + \Lambda_j R_0^{-2} U_{j,R_0}
	\geq R_0^{-n(p_j-1)} U_{j+1,R_0}^{p_j},
	\label{eq:3.5}
	\end{align}%
where
	\[
	\Lambda_k = 2 \lambda,
	\quad \Lambda_{j} = p_j \Lambda_{j+1} \ (j=1,\ldots,k-1).
	\]
Put $\widetilde U_{j} = e^{\Lambda_j R_0^{-2} t} U_{j,R_0}$
for $1 \leq j \leq k$.
Then, by \eqref{eq:3.5}, for $1 \leq j \le k-1$, we have
	\begin{align*}
	\widetilde U_j'
	&= e^{\Lambda_j R_0^{-2} t}
	\left( U_{j,R_0}' + \Lambda_j R_0^{-2}U_{j,R_0} \right)\\
	&\geq R_0^{-n(p_j-1)}
	e^{\Lambda_j R_0^{-2} t} U_{j+1,R_0}^{p_j}\\
	&= R_0^{-n(p_j-1)} e^{-(p_j\Lambda_{j+1}-\Lambda_j) R_0^{-2} t}
	\widetilde U_{j+1}^{p_j} \\
	&= R_0^{-n(p_j-1)} \widetilde U_{j+1}^{p_j},
	\end{align*}
and
	\begin{align*}
	\widetilde U_k'
	&= e^{\Lambda_k R_0^{-2} t}
	\left( U_{k,R_0}' + \Lambda_k R_0^{-2}U_{k,R_0} \right)\\
	&\geq R_0^{-n(p_k-1)}
	e^{\Lambda_k R_0^{-2} t} U_{1,R_0}^{p_k}\\
	&= R_0^{-n(p_k-1)} e^{-(p_k\Lambda_{1}-\Lambda_k) R_0^{-2} t} \widetilde U_{1}^{p_k}.
	\end{align*}
Then we apply Proposition \ref{Proposition:2.1}
with $f_j = \widetilde U_j$,
$\widetilde C_j = R_0^{-n (p_j-1)}$,
$\widetilde \lambda = (p_k\Lambda_{1}-\Lambda_k) R_0^{-2}$.
Indeed, we first
remark that $\widetilde U_{2}(0) > 0$ holds
since we assume that $u_{0,2} \geq 0$
and $u_{0,2} \not \equiv 0$
(see also Remark 1.2 (i) and \eqref{eq:3.3}).
Next, we check the condition \eqref{eq:2.9}.
We note that
$p_k\Lambda_{1}-\Lambda_k = 2 \lambda ( \prod_{j=1}^kp_j -1) = 2 \lambda (P_1-Q_1 -1)$,
and hence,
	\begin{align}
	\label{eq:3.6}
	\frac{L_1}{Q_1}
	= \frac{\widetilde{\lambda}}{Q_1}
	\prod_{m =1}^{k-1} p_{m}
	= 2 \lambda R_0^{-2}  \frac{P_1-Q_1 -1}{Q_1} 
	\prod_{m =1}^{k-1} p_{m}
	= 2 \lambda R_0^{-2} \alpha_1^{-1} \prod_{\ell =1}^{k-1} p_{m}.
	\end{align}
Here, we have used \eqref{eq:2.5}.
From here, $C$ denotes general constants independent of $R_0$.
If $k = 2$, $A_2$ is computed as
	\begin{align}
	\label{eq:3.7}
	A_{2}
	= P_{2}^{-1} \widetilde C_{2}
	= C R_0^{-n(p_2-1)}
	= C R_0^{-n(P_2-Q_2-1)},
	\end{align}
otherwise,
\begin{align}
\label{eq:3.8}
	A_{2}
	&= P_{2}^{-1} \widetilde C_{2}
	\prod_{h=0}^{k-3} (P_{k-h}^{-1} \widetilde C_{k-h})^{\prod_{m=2}^{k-h-1}p_m}\\
	\nonumber
	&= C
	R_0^{-n(p_2-1 + \sum_{h=0}^{k-3} (p_{k-h} -1 )\prod_{m=0}^{k-3-h}p_{m+2})}\\
	\nonumber
	&= C
	R_0^{-n(\prod_{m=0}^{k-2}p_{m+2} - 1)}\\
\nonumber
	&= C
	R_0^{-n(P_{2}-Q_{2}-1)}.
\end{align}
Here, we have used \eqref{eq:2.4}.
We recall that by \eqref{eq:2.6} and \eqref{eq:2.8},
	\[
	\widetilde C
	=C \widetilde \lambda^{-1} \widetilde C_1^{1/Q_1} A_{2}^{p_1/Q_1}.
	\]
From this and $\widetilde{\lambda}^{-1} = CR_0^2$,
RHS of \eqref{eq:2.9} is calculated as
	\begin{align}
	\label{eq:3.9}
	\widetilde C^{-\alpha_2} A_2^{\alpha_2/(\alpha_1 P_2)}
	\widetilde C_1^{-\alpha_2 Q_2/(\alpha_1 P_2)}
	= C \Big(
	R_0^{-2} \widetilde C_1^{-1/Q_1 - Q_2/(\alpha_1 P_2)} A_2^{1/(\alpha_1 P_2) - p_1/Q_1}
	\Big)^{\alpha_2}
	\end{align}
Let us further compute the RHS.
We also directly obtain
	\begin{align*}
	-\frac{1}{Q_1} - \frac{Q_2}{\alpha_1 P_2}
	&= -\frac{1}{Q_1} - \frac{Q_2}{P_2} \bigg( \frac{P_1-1}{Q_1} - 1 \bigg)\\
	&= -\frac{1}{Q_1} - \frac{Q_2}{P_2} \bigg( \frac{p_1 P_2}{Q_1} - 1 \bigg)\\
	&= -\frac{1}{Q_1} - \frac{Q_1-1}{Q_1} + \frac{Q_2}{P_2}
	= \frac{Q_2-P_2}{P_2},\\
	\frac{1}{\alpha_1 P_2} - \frac{p_1}{Q_1}
	&= \frac{P_1-Q_1-1}{Q_1 P_2} - \frac{p_1}{Q_1}
	= -\frac{1}{P_2}.
	\end{align*}
Since $\widetilde C_1 = C R^{-n(p_1-1)}$,
	\begin{align}
	\label{eq:3.10}
	\widetilde C_1^{-1/Q_1 - Q_2/(\alpha_1 P_2)}
	&= C R_0^{-n(p_1-1)(Q_2-P_2)/P_2}
	= C R_0^{n(P_1-Q_1 - (P_2-Q_2) )/P_2}.
	\end{align}
Moreover, \eqref{eq:3.7} and \eqref{eq:3.8} imply
	\begin{align}
	\label{eq:3.11}
	A_2^{1 /(\alpha_1 P_2) - p_1/Q_1}
	&= C R_0^{n(P_2-Q_2-1)/P_2}.
	\end{align}
From \eqref{eq:3.9}--\eqref{eq:3.11},
that the LHS of \eqref{eq:3.9} is simply expressed as
	\begin{align*}
	\Big( \widetilde C^{-1} A_2^{1/(\alpha_1 P_2)} \widetilde C_1^{- Q_2/(\alpha_1 P_2)}
	\Big)^{\alpha_2}
	&= C \Big( R_0^{-2 + n(P_1-Q_1-1)/P_2} \Big)^{\alpha_2}
	= C R_0^{-2 \alpha_2 +n}.
	\end{align*}
Here, in the constant $C$ on the RHS depends only on $(p_1, \ldots, p_k)$.
Let us rewrite the identity above as
	\begin{align}%
	\label{eq:3.12}
	\Big(
	\widetilde C^{-1} A_2^{1/(\alpha_1 P_2)} \widetilde C_1^{- Q_2/(\alpha_1 P_2)}
	\Big)^{\alpha_2}
	= ( \overline{C}_1  R_0^{-1} )^{2 \alpha_2 -n}.
	\end{align}%
This and \eqref{eq:3.3} imply
	\[
	\widetilde C^{-1} A_2^{1/(\alpha_1 P_2)} \widetilde C_1^{-Q_2/(\alpha_1 P_2)}
	U_{2,R_0}(0)^{-1/\alpha_2}
	= \left( \overline{C}_1 R_0^{-1} \right)^{2-n/\alpha_2} U_{2,R_0}(0)^{-1/\alpha_2}
	= \frac 1 2.
	\]
Thus, the assumption \eqref{eq:2.9} holds, and
we can apply Proposition \ref{Proposition:2.1}.
The assertion \eqref{eq:2.10} of Proposition \ref{Proposition:2.1} with \eqref{eq:3.6}
lead to the estimate \eqref{eq:3.1} with $T_0$ satisfying
	\begin{align*}
	T_0
	&= - L_1^{-1}Q_1 \log
	\left( 1 - \widetilde C^{-1} A_2^{1/(\alpha_1 P_2)} \widetilde C_1^{-Q_2/(\alpha_1 P_2)}
	U_{2,R_0}(0)^{-1/\alpha_2} \right)\\
	&\leq C R_0^{2} \\
	&\leq C U_{2,R_0}(0)^{-1/(\alpha_2 - n/2)}.
\end{align*}
Here, the first identity follows from the assertion of Proposition \ref{Proposition:2.1},
the second inequality is due to \eqref{eq:3.6},
and the third inequality is due to \eqref{eq:3.3}.
This completes the proof.
\end{proof}

\section{Lower bound of the lifespan}
In this section, we give the proof of Proposition \ref{Proposition:1.3}.
The global existence for the case when $\alpha_{\mathrm{max}} < n/2$
may be shown by the argument of \cite{EsHe91,Um03,OgTa09}.
However, for reader's convenience, we give a proof here.
By Proposition \ref{Proposition:1.1}, we construct the local solution
$u \in C([0,T_m); L^1(\mathbb R^n) \cap L^{\infty}(\mathbb R^n))$
of the
integral equation
	\[
	u_j(t) = e^{t\Delta} u_{0,j} + \int_0^t e^{(t-\tau)\Delta} |u_{j+1}|^{p_j}\,d\tau.
	\]
Let $l_j \le n/2$ determined later and we define
	\begin{align*}%
	M(t) &:= \sup_{\tau \in \lbrack 0,t)}
	\sum_{j=1}^k \left\{ (1+\tau)^{l_j} \| u_j(\tau) \|_{L^{\infty}}
	+ (1+\tau)^{l_j-n/2} \| u_j(\tau) \|_{L^1} \right\}.
	\end{align*}%
It is well known that for $1 \leq p \leq q \leq \infty$,
	\begin{align}%
	\| e^{t\Delta} f \|_{L^q}
	\le C t^{-\frac{n}{2}\left(\frac{1}{p}-\frac{1}{q}\right)} \| f \|_{L^p}.
	\label{eq:4.1}
	\end{align}%
By using \eqref{eq:4.1},
\begin{align*}%
	&(1+t)^{l_j-n/2} \| u_{j}(t) \|_{L^{1}}\\
	&\le C (1+t)^{l_j-n/2} \| u_{0,j} \|_{L^{1}}
		+ C (1+t)^{l_j-n/2} \int_0^{t} \| u_{j+1}(\tau) \|_{L^{p_{j}}}^{p_{j}}\,d\tau \\
	&\le C (1+t)^{l_j-n/2} \| u_{0,j} \|_{L^{1}} \\
	&\quad + C (1+t)^{l_j-n/2} \int_0^t (1+\tau )^{-l_{j+1}p_j + n/2}M(\tau)^{p_j} \,d\tau \\
	&\le C (1+t)^{l_j-n/2} \| u_{0,j} \|_{L^{1}}
		+ C L_j (t) M(t)^{p_j},
\end{align*}%
where
\begin{align*}
	L_j(t) = \begin{cases}
		(1+t)^{l_j-n/2}&(-l_{j+1} p_j + n/2 < -1),\\
		(1+t)^{l_j-n/2} \log (1+t) &(-l_{j+1} p_j + n/2 = -1),\\
		(1+t)^{l_j - l_{j+1}p_j +1} &(-l_{j+1} p_j + n/2 > -1)
	\end{cases}
\end{align*}
and we have used the fact that,
by the interpolation and Young inequality
	\begin{align*}
	&\| u_{j+1}(\tau) \|_{L^{p_j}}^{p_j}\\
	&\leq \| u_{j+1}(\tau) \|_{L^\infty}^{p_{j}-1} \| u_{j+1}(\tau) \|_{L^1}\\
	&= (1+\tau)^{-p_j l_{j+1} + n/2}
	((1+\tau)^{l_{j+1}} \| u_{j+1}(\tau) \|_{L^\infty})^{p_j-1}
	((1+\tau)^{l_{j+1}-n/2} \| u_j(\tau) \|_{L^1})\\
	&\leq (1+\tau)^{-p_j l_{j+1} + n/2} M(\tau)^{p_j}.
	\end{align*}

Next, we consider the estimate for $\| u_j (t) \|_{L^{\infty}}$.
First, for $0<t \le 1$, we apply $L^{\infty}$-$L^{\infty}$ estimate to obtain
	\begin{align*}
	\| u_j(t) \|_{L^{\infty}}
	&\le C \| u_{0,j} \|_{L^{\infty}}
		+ C \int_0^t \| u_{j+1}(\tau) \|_{L^{\infty}}^{p_j} \,d\tau \\
	&\le C \| u_{0,j} \|_{L^{\infty}}
		+ C \int_0^1 (1+\tau)^{-l_{j+1} p_j} M(\tau)^{p_j} \,d\tau \\
	&\le C \| u_{0,j} \|_{L^{\infty}} + CM(t)^{p_j}.
	\end{align*}
For $t \ge 1$, we apply
$L^1$-$L^{\infty}$ estimate to obtain
\begin{align*}
	(1+t)^{l_j}\| u_{j}(t) \|_{L^{\infty}}
	&\le C(1+t)^{l_j} t^{-n/2} \| u_{0,j} \|_{L^1} \\
	&\quad + C (1+t)^{l_j}
	\int_0^{t/2} (t-\tau)^{-n/2} \| | u_{j+1}(\tau) |^{p_j} \|_{L^1} \,d\tau \\
	&\quad + C (1+t)^{l_j}
	\int_{t/2}^t \| u_{j+1}(\tau) \|_{L^\infty}^{p_j} \,d\tau \\
	&\le C(1+t)^{l_j -n/2}\| u_{0,j} \|_{L^1} \\
	&\quad + C (1+t)^{l_j -n/2} \int_0^{t/2} (1+\tau)^{-l_{j+1}p_j +n/2} M(\tau)^{p_j}\,d\tau \\
	&\quad + C (1+t)^{l_j} \int_{t/2}^t (1+\tau)^{-l_{j+1}p_j} M(\tau)^{p_j} \, d\tau \\
	&\le  C(1+t)^{l_j -n/2}\| u_{0,j} \|_{L^1}
		+ C L_j(t) M(t)^{p_j}.
\end{align*}
We remark that, the definition of $L_j(t)$ implies that for any $t >0$,
	\[
	(1+t)^{l_j} \int_{t/2}^t (1+\tau)^{-l_{j+1}p_j} d \tau
	\leq C (1+t)^{l_j - l_{j+1}p_j + 1}
	\leq C L_j(t).
	\]
Therefore, we conclude
	\begin{align}%
	\label{eq:4.2}
	M(t) &\le C_0 \| u_0 \|_{L^1 \cap L^{\infty}}
		+ C_1 \max_{1\le j \le k} (L_j(t) M(t)^{p_j})
	\end{align}%
with some constants
$C_0, C_1 > 0$.

Now, we determine $l_j \ (1\le j \le k)$ in the following way.
Let 
$\mathbf{1} = {}^t\!(1,\ldots,1)$,
$l = {}^t\!(l_1,\ldots, l_k)$,
$\alpha = {}^t\!(\alpha_1,\ldots,\alpha_k)$.

\noindent
{\bf Case 1}: When $\alpha_{\max} < n/2$,
we take $\varepsilon > 0$
so that
$(1+\varepsilon) \alpha_{\max} < n/2$.
Then, we determine $l$ by the relation
\begin{align*}
	l_j - l_{j+1}p_j + 1 = - \varepsilon,
\end{align*}
namely,
\begin{align*}
	l = (1+\varepsilon) \left( \mathcal{P} - I \right)^{-1} \mathbf{1} = (1+\varepsilon) \alpha.
\end{align*}
Then, it is obvious that
$l_j < n/2$ for $1\le j \le k$.
Therefore, by the definition of
$L_j(t)$,
in any case we have $L_j(t) \le C$
with some constant $C>0$ independent of $t \ge 0$.
Thus, from \eqref{eq:4.2}, we have the a priori estimate
\begin{align*}
	M(t) \le C_0 \| u_0 \|_{L^1 \cap L^{\infty}}
		+ C_1' \max_{1\le j \le k} M(t)^{p_j},
\end{align*}
which enables us to prove the small data global existence of the solution.
Namely, in this case $T_m = \infty$ holds for small initial data.

\noindent
{\bf Case 2}: When $\alpha_{\max} > n/2$,
we determine $l$ by the relation
	\[
	l_j - l_{j+1} p_j +1 = \left( \alpha_{\max} - n/2 \right) (p_j -1) \quad (1 \le j \le k),
	\]
that is,
	\[
	l = \alpha - \left( \alpha_{\max} - n/2 \right) \mathbf{1}.
	\]
In this case, again,
	\begin{align}
	l_j = \alpha_j - \alpha_{\max} + n/2 \le n/2.
	\label{eq:4.3}
	\end{align}
Moreover, we remark that for any $j$,
	\begin{align}
	-l_{j+1} p_j + n/2 > -1.
	\label{eq:4.4}
	\end{align}
Indeed, if there exists some $j_0$ such that
	\[
	-l_{j_0+1} p_{j_0} + n/2 \leq -1,
	\]
then
	\[
	0
	\geq l_{j_0} - n/2
	\geq l_{j_0} - p_{j_0} l_{j_0+1} + 1
	= (\alpha_{\mathrm{max}} - n/2) (p_{j_0} -1) \geq 0.
	\]
Therefore
$p_{j_0} = 1$, $l_{j_0} = n/2$, and $l_{j_0+1} \geq n/2 +1$,
which contradicts \eqref{eq:4.3}.
By \eqref{eq:4.4}, for any $1 \leq j \leq k$,
	\[
	L_j(t)
	= (1+t)^{(\alpha_{\max} -n/2)(p_j-1)}.
	\]
Therefore, by \eqref{eq:4.2} we conclude
	\[
	M(t)
	\le C_0 \| u_0 \|_{L^1 \cap L^{\infty}}
		+ C_1 \max_{1\le j \le k}
			(1+t)^{(\alpha_{\max} -n/2)(p_j-1)} M(t)^{p_j}.
	\]
We take again
$T_1$ as the smallest time such that
$M(t) = 2C_0 \| u_0 \|_{L^1 \cap L^{\infty}}$
(we note that if such a time does not exist, we have $T_m = \infty$).
Then, substituting $t=T_1$ in the above inequality, we have
	\[
	C_0 \| u_0 \|_{L^1 \cap L^{\infty}}
	\leq C_1 \max_{1\le j \le k}
		(1+T_1)^{(\alpha_{\max} -n/2)(p_j-1)}
			\left( 2 C_0 \| u_0 \|_{L^1 \cap L^{\infty}} \right)^{p_j},
	\]
which and smallness of $u_0$ imply that
	\[
	T_m > T_1 \geq C \| u_0 \|_{L^1 \cap L^{\infty}}^{-1/(\alpha_{\max} - n/2)}.
	\]
This completes the proof.

\refstepcounter{section}

\section*{Acknowledgments}
The authors would like to express their hearty thanks to Dr. K\={o}dai Fujimoto
for valuable discussions.
This work was partly supported by JSPS KAKENHI Grant Numbers
16J30008, 15K17571, and 16K17625.



\end{document}